\documentclass[12pt,twoside]{article}
\usepackage{amssymb,amsthm}
\usepackage[namelimits,sumlimits,fleqn]{amsmath}
\usepackage{setspace}
\usepackage[ hmargin={2.2cm,2.4cm}, top=36mm,
    a4paper,pdftex, textheight=230mm,
    headheight=20mm, headsep=10mm, bindingoffset=0.9cm]
    {geometry}
\usepackage[small, margin=20pt]{caption}
\usepackage[usenames,dvipsnames]{color}
\definecolor{darkblue}{rgb}{0.0,0.0,0.6}
\usepackage{float}\restylefloat{figure}
\usepackage{url}
\usepackage[backref=page]{hyperref}
\hypersetup{colorlinks=true,breaklinks=true,
            linkcolor=darkblue,urlcolor=darkblue,
            anchorcolor=darkblue,citecolor=darkblue}

\usepackage{fancyhdr}
\allowdisplaybreaks[4]

\title{On the Number of Dot Products Determined by a Large Set and One of its Translates in Finite Fields}
\author{Giorgis Petridis}
\date{}

\theoremstyle{plain}
\newtheorem{theorem}{Theorem}[section]
\newtheorem{lemma}[theorem]{Lemma}

\newtheorem{corollary}[theorem]{Corollary}

\theoremstyle{definition}

\newtheorem*{acknowledgement}{Acknowledgement}

\theoremstyle{definition}
\newtheorem*{unremark}{Remark}

\newcommand{\F}{\mathbb{F}_q} 
\newcommand{\R}{\mathbb{R}} 
\newcommand{\f}{\frac} 
\newcommand{\cd}{\cdot} 
\newcommand{\sm}{\setminus} 

\renewcommand*{\backref}[1]{}
\renewcommand*{\backrefalt}[4]{%
    \ifcase #1 (Not cited.)%
    \or        (Cited on page~#2.)%
    \else      (Cited on pages~#2.)%
    \fi}

\begin{document}

\onehalfspacing

\pagenumbering{arabic}

\setcounter{section}{0}

\bibliographystyle{plain}

\maketitle

\begin{abstract}
Let $E \subseteq \F^2$ be a set in the 2-dimensional vector space over a finite field with $q$ elements, which satisfies $|E| > q$. There exist $x,y \in E$ such that $|E \cd (y-x)| > q/2.$ In particular, $(E+E) \cd (E-E) = \F$.
\end{abstract}

\section[Introduction]{Introduction}
\label{Introduction}

\let\thefootnote\relax\footnotetext{The author is supported by the NSF DMS Grant 1500984.}

The question of determining a lower bound on the cardinality of a set $E \subseteq \F^2$ in a 2-dimensional vector space over a finite field $\F$ with $q$ elements, which guarantees that the set of dot products determined by $E$ has cardinality strictly greater than $q/2$, goes back at least as far as a paper of Hart and Iosevich \cite{Hart-Iosevich2008}. By the \emph{set of dot products determined by $E$} we mean the following subset of $\F$:
\[
E \cd E = \{ u \cd v : u, v \in E\}.
\]
Hart and Iosevich proved that if $|E| > q^{3/2}$, then $E \cd E = \F^*:=\F\sm\{0\}$, with the exponent of $3/2$ in the lower bound for $|E|$ being essentially sharp (see Corollary 2.4 in \cite{HIKR2011}). 
For similar results in higher dimensions see \cite{CEHIK2012}.

An analogous result was proved in the context of geometric measure theory by Erdo\u{g}an, Hart, and Iosevich \cite{EHI2013}. The authors showed that if a planar set $E \subset \R^2$ has Hausdorff dimension $\dim_H(E) > 3/2$, then the set $E \cd E \subset \R$ has positive Lebesgue measure. The result can loosely be interpreted as saying that the set of dot products determined by a ``large'' planar set is ``large''. 

The same argument shows that if $\dim_H(E) > 1$, there exists $x \in E$ such that the set 
\[
E \cd (E -x) = \{ u \cd (v -x) : u,v \in E\}
\]
has positive Lebesgue measure. 

Recently, in a breakthrough paper on the Falconer conjecture, Orponen established a similar result. If $E \subset \R^2$ is an Ahlfors-David regular planar set with Hausdorff dimension $\dim_H(E) \geq 1$, then there exists $x \in E$ such that $E\cd (E -x)$ has packing dimension equal to 1 \cite{Orponen2015}.

We prove an analogous result in the context of 2-dimensional vector spaces over finite fields. In fact our result concerns ``pinned'' dot products.

\begin{theorem} \label{PinTransDotProd}
Let $E \subseteq \F^2$ be a set in the 2-dimensional vector space over a finite field with $q$ elements. Suppose that $|E| > q$. There exist $x,y \in E$ such that
\[
|E \cdot (y-x)| = |\{ u \cd (y-x) : u \in E \}| > \f{q}{2}.
\]
\end{theorem}

The lower bound is sharp. 
If $q=p^2$ is the square of a prime $p$ and $E$ is the Cartesian product of a subfield isomorphic to $\mathbb{F}_p$, then $|E|=p^2=q$ and $|E \cdot (y-x)| = p = \sqrt{q}$ for all $x,y \in E$. 

It is likely that one could prove the existence of $x \in E$ such that $|E \cd (E-x)| > q/2$ provided that $|E| = \Omega(\sqrt{q \log(q)})$ by combining Theorem~2 in \cite{IRZ2015} with Theorem 2.6 in \cite{HIKR2011}. Our proof is different. 

The advantage of proving a result about ``pinned'' dot products is illustrated in the following corollary.

\begin{corollary} \label{E+E dot E-E}
Let $E \subseteq \F^2$ be a set in the 2-dimensional vector space over a finite field with $q$ elements. Suppose that $|E| > q$. There exist $x,y \in E$ such that
\[
\F = (E+E) \cd (y-x) = \{ (u+v) \cd (y-x) : u,v \in E\}.
\]
In particular 
\[
\F = (E+E) \cd (E-E) = \{ (u+v) \cd (z-w) : u,v,z,w \in E\}.
\]
\end{corollary}


When $E = A \times A$ for a symmetric set $A$ (that is, $-A=A$), we recover a result of Glibichuk from \cite{Glibichuk2006} that asserts that the 8-fold sumset of $AA$ is the whole of $\F$ provided that $|A| > \sqrt{q}$:
\begin{equation} \label{Glibichuk}
\F = \{a_1 a_2 + a_3 a_4 + \dots + a_{15} a_{16} : a_1,\dots, a_{16} \in A \}.
\end{equation}

We conclude the introduction with a short detour on a similar question.  As mentioned above, Hart, Iosevich, Koh, and Rudnev showed in \cite{HIKR2011} that the lower bound $q^{3/2}$ is essentially optimal if we require $E \cd E$ to be a positive proportion of $\F$. 
Their example is the union of half-lines and so tells us little about the case where $E = A \times A$ is a Cartesian product. When $E = A \times A$ is a Cartesian product, the set of dot products takes the form $AA+AA$. Bounding from below $|AA+AA|$ has received much attention in the literature and is worth summarising. 
\begin{enumerate}
\item Hart and Iosevich proved in \cite{Hart-Iosevich2008} that 
\[
|AA+AA| = \Omega\left(\min\left\{q, \f{|A|^3}{q}\right\}\right).
\]
A more precise result they proved is that $|AA+AA|>q/2$ when $|A| > q^{2/3}$.
\item When $A$ is a multiplicative subgroup of $\F^*$ we have $AA+AA=A+A$. Heath-Brown and Konyagin proved in \cite{Heath-Brown-Konyagin2000} the following lower bound for multiplicative subgroups of prime order fields (we replace $q$ by $p$ for clarity) via an elegant application of Stepanov's method 
\[
|A+A| = \Omega\left( \min\{p, |A|^{3/2}\}\right).
\]
Shkredov and Vyugin improved the lower bound at the cost of the additional assumption $|A| = O(p^{1/2})$ in \cite{Shkredov-Vyugin2012}
\[
|A+A| = \Omega(|A|^{5/3} \log(|A|)^{-1/2}).
\] 
Shkredov recently proved in \cite{Shkredov2016Tripling} that under the same hypothesis $|A| = O(p^{1/2})$
\[
|A+A+A| = \Omega(|A|^2/\log(|A|)).
\]
This corresponds to  $E \cd E$ for $E = A \times A \times A \subset \F^3$.
\item Rudnev in \cite{Rudnev} generalised the Heath-Brown and Konyagin lower bound to all sets in a prime order field 
\[
|AA+AA| = \Omega\left( \min\{p, |A|^{3/2}\}\right).
\]
\end{enumerate} 
It is likely that $AA+AA$ is at least, say, half of $\F^*$ for any set $A$ of cardinality a sufficiently large multiple of $\sqrt{q}$, at least for prime $q$.

In the next section we offer an overview of the proof of Theorem~\ref{PinTransDotProd} and prove the necessary lemmata. In the final section we prove Theorem~\ref{PinTransDotProd} and Corollary~\ref{E+E dot E-E}.

\begin{acknowledgement}
The author would like to than Alex Iosevich, Jonathan Pakianathan, and Misha Rudnev for helpful conversations. Misha Rudnev suggested using Lemma~\ref{IMP}, which simplified the original argument and yielded the optimal lower bound on $|E|$ in Theorem~\ref{PinTransDotProd}.
\end{acknowledgement}

\section{First and second moment calculations}
\label{Lemmata}

The proof of Theorem~\ref{PinTransDotProd} consists of two distinct parts. Before describing them, let us set up some notation. To a direction $\theta \in \F \cup \{\infty\}$ we associate the direction vector $v_\theta$, which equals $(1,\theta)$ if $\theta \in \F$ and $(0,1)$ if $\theta =\infty$. A \emph{direction $\theta \in \F \cup \{\infty\}$ is determined by a set $E \subseteq \F^2$} if there exists $\lambda \in \F^*$ such that $ \lambda v_\theta \in E$. 

The two steps of the proof are as follows.
\begin{enumerate}
\item Let $E,F$ be two sets in $\F^2$. Suppose that $|E| > q$ and that $F$ determines all directions in $\F\cup \{\infty\}$. There exists $v \in F$ such that
\[
|E \cd v| = |\{ u \cd v : u \in E \}| > \f q 2.
\] 
\item Let $E$ be a set in $\F^2$. Suppose that $|E| > q$. Every direction in $\F\cup\{\infty\}$ is determined by $E$. This is a result of Iosevich, Morgan, and Pakianathan (Theorem~2 in \cite{IMP2011}). It was proved in the case where $E=A \times B$ is a Cartesian product by Bourgain, Glibichuk, and Konyagin in \cite{BGK2006}.
\end{enumerate}

The first step can be thought of as a discrete version of a classical theorem of Marstrand about projections in Euclidean space \cite{Marstrand1954} and will be proved by a simple a second moment calculation. The second step 
will be proved by an application of the pigeonhole principle.

\subsection{The second moment of a point-line incidence function}
\label{Incidences}

Given two sets $E, F \in \F^2$ and $t \in F$, there exist $u \in E$ and $v \in F$ such that $u \cd v = t$ precisely when $E$ is incident to the line $\{ w \in \F^2: w \cd v = t \}$. 
To motivate the proof of the first step outlined above, suppose for a contradiction that $F$ determines $\Omega(q)$ directions and that $|\F \sm (E \cd F)| = \Omega(q)$. It follows that $|E|$ is not incident to $\Omega(q^2)$ lines. To show that this is impossible when $|E| = \Omega(q)$ we will prove that ``$E$ is incident to most lines roughly the expected number of times''.

To this end we denote by $i(\ell)$ the number of incidences of a line $\ell \subseteq \F^2$ with $E$
\begin{equation}\label{i}
i_E(\ell) = i(\ell) = | \ell \cap E|.
\end{equation}

There are $q(q+1)$ lines in $\F^2$ ($q+1$ possible slopes and $q$ possible $y$-axis intercepts) and $|E| (q+1)$ point-line incidences between $E$ and the set of all lines with slope in $\F$ (there are $q+1$ lines incident to each point of $E$). Therefore, on average a line is incident to $|E|/q$ points from $E$. We show that this is typically a very good estimate by obtaining an exact expression for the second moment of $i(\ell)$.

\begin{lemma} \label{var i}
Let $E\subseteq F_q^2$ and $i$ be the function defined in \eqref{i}. 
\[
\sum_\ell i(\ell)^2 = |E|^2 +  q|E|.
\]
The sum is over all lines in $\F^2$. 
\end{lemma}

A few remarks before we prove this lemma. It asserts in probabilistic language that $\mathrm{Var}[i] \leq \mathrm{E}[i]$ and is based on the fact a collection of lines in $\F^2$ is a pseudorandom collection of subsets. It is a generalisation of Lemma~2.1 of Bourgain, Katz, and Tao from \cite{BKT2004}. The authors considered the case where $E = A \times B$ is a Cartesian product. It is also very close to a point-line incidence theorem of Vinh \cite{Vinh2011}. 

\begin{proof}[Proof of Lemma~\ref{var i}]
Sums are over all lines in $\F^2$.  We denote by $\ell$ the characteristic function of a line $\ell$.
\begin{align*}
\sum_\ell i(\ell)^2 
& = \sum_\ell \left( \sum_{v \in E} \ell(v) \right)^2 \\
& = \sum_\ell  \sum_{v,v' \in E} \ell(v) \ell(v') \\
& = \sum_{v \in E} \sum_\ell \ell(v) + \sum_{v\neq v' \in E} \sum_{\ell} \ell(v) \ell(v')\\
& =  |E| (q+1) + |E| (|E| - 1) \\
& = |E|^2 + q |E|.
\end{align*}
In the penultimate line we used the facts that $q+1$ lines are incident to a point and that two distinct points determine a unique line.

\end{proof}

\subsection{Many directions give a good vector to project on}

Deducing the first step outlined at the beginning of the section is now only a matter of labelling lines, averaging, and and applying the Cauchy-Schwarz inequality.

\begin{corollary} \label{Many directions}
Let $E,F$ be two sets in $\F^2$. Suppose that $|E| > q$ and that $F$ determines all directions in $\F\cup\{\infty\}$. There exists $v \in F$ such that
\[
|E \cd v| = |\{ u \cd v : u \in E \}| > \f q 2.
\]
\end{corollary}

\begin{proof}
We show that there exists $v \in F$ with the property that $E$ is approximately equidistributed on the lines orthogonal to $v$. The conclusion follows in a straightforward manner by the Cauchy-Schwarz inequality. 

For each direction $\theta \in \F\cup\{\infty\}$ let $v_\theta \in F$ be the vector that determines $\theta$ described at the beginning of the section. Next label by $\ell_{\theta,t}$ the line $\{ w \in \F^2 : w \cd v_\theta = t\}$. Lemma~\ref{var i} implies
\[
\sum_{\theta \in \F\cup\{\infty\}} \sum_{t \in \F}  i(\ell_{\theta,t})^2  =  \sum_\ell i(\ell)^2 = |E|^2 + q |E| < 2|E|^2.
\]
Therefore, noting that $i(\ell_{\theta,t})=0$ unless $t \in v_\theta \cdot E$, there exists $\theta \in \F$ such that 
\[
\sum_{t \in v_\theta \cdot E}  i(\ell_{\theta,t})^2 = \sum_{t \in \F}  i(\ell_{\theta,t})^2  < \frac{2|E|^2}{q}.
\]
By the Cauchy-Schwarz inequality
\[
 \frac{2|E|^2}{q} > \sum_{t \in v_\theta \cdot E}  i(\ell_{\theta,t})^2 \geq \frac{\left( \sum_{t \in v_\theta \cdot E}  i(\ell_{\theta,t}) \right)^2}{|v_\theta \cdot E|} = \frac{|E|^2}{|v_\theta \cdot E|},
\]
which implies $|v_\theta \cdot E| > q/2$.
Finally, $v_\theta = \lambda v$ for some $v \in F$ and $\lambda \in \F^*$ and so 
\[
|E \cd v| = |\lambda(E \cd v)| = |E \cd (\lambda v)| = |E \cd v_\theta| > q/2.
\]
\end{proof}

\subsection{A result of Iosevich, Morgan, and Pakianathan}

We now turn our attention to the second step outlined at the beginning of the section and prove the result of Iosevich, Morgan, and Pakianathan. The special case when $E = A \times B$ is a Cartesian product was proved by Bourgain, Glibichuk, and Konyagin (statement (9) in the proof of Theorem~3 in \cite{BGK2006}). For completeness, we provide a proof communicated by Rudnev that is along the lines of the Bourgain, Glibichuk, and Konyagin. It should also be noted that considerations of Alon at the end of Section 4 of \cite{Alon1986} imply a similar result. 

\begin{lemma}[Iosevich, Morgan, and Pakianathan, Theorem~2 in \cite{IMP2011}] \label{IMP}
Let $E$ be a set in $\F^2$. Suppose that $|E| > q$. The difference set $ E-E = \{u -w: u,v \in E\}$ 
determines all directions in $\F \cup\{\infty\}$.
\end{lemma}
\begin{proof}
Let $\theta$ be a direction in $\F\cup\{\infty\}$, $v_\theta$ be a vector determining it, and $\ell_\theta = \{\mu v_\theta : \mu \in \F\}$ a line with direction $\theta$.

Note that $|E| |\ell| > q^2$ and so the pairwise products $u + \mu v_\theta$ with $v \in E$ and $\mu \in \F$ cannot all be distinct. Therefore there exist distinct $u,w \in E$ and $\mu, \mu' \in \F$ such that $u + \mu v_\theta = w + \mu' v_\theta$ or $u-w = (\mu' - \mu) v_\theta$. In other words, $E-E$ determines the direction $\theta.$
\end{proof}

\section{Proofs of Theorem~1.1 and Corollary~1.2}
\label{Proofs}


\begin{proof}[Proof of Theorem~\ref{PinTransDotProd}]
The hypothesis $|E|> q$ and Lemma~\ref{IMP} every direction is determined by $E-E$. Corollary~\ref{Many directions} and the hypothesis $|E| > q$ now imply there exists $y -x \in E-E$ such that $|E \cd (y-x)| > q/2$.
\end{proof}

\begin{proof}[Proof of Corollary~\ref{E+E dot E-E}]
Let $S = E \cd (y-x) \subseteq \F$. $|S| > q/2$ and so $S+S = \F$ (for each $\xi \in \F$, the sets $S$ and $\xi - S$ must intersect because the sum of their cardinalities exceeds $q$). Therefore $ \F = E \cd (y-x) + E \cd (y-x) = (E + E) \cd (y-x).$
\end{proof}

When $E = A \times A$ we instantly get equality \eqref{Glibichuk}
.
\begin{align*}
(E+E) \cd (E-E) 
& = [(A+A) \times (A+A)] \cd [(A-A) \times (A-A)] \\
& = (A+A)(A-A) + (A+A) (A-A) \\
& \subseteq AA - AA + AA -AA + AA - AA + AA - AA.
\end{align*}

\begin{unremark}
The arguments presented in this note highlight the importance of the direction set of $E$ in dot-product related questions. They suggest that when looking to bound $|AA+AA|$ from below, the case where $E = A \times A$ for a multiplicative subgroup $A \subseteq \F^*$ might, in a sense, be extremal.
\end{unremark}


\phantomsection

\addcontentsline{toc}{section}{References}

\bibliography{//Users/georgiopetridis//Dropbox/bib/all}

\hspace{20pt} Department of Mathematics, University of Georgia, Athens, GA 30602, USA.

\hspace{20pt} \textit{Email address}: \href{mailto:giorgis@cantab.net}{giorgis@cantab.net}

\end{document}